\documentclass[12pt]{amsart}
\usepackage{amssymb,amsmath,amsfonts,latexsym}
\usepackage{bm}

\usepackage{array,graphics,color}
\newcommand{\newsection}[1]{\setcounter{equation}{0} \section{#1}}
\numberwithin{equation}{section}

\newtheorem{propn}{Proposition}[section]
\newtheorem{thm}[propn]{Theorem}
\newtheorem{lemma}[propn]{Lemma}
\newtheorem{cor}[propn]{Corollary}
\newtheorem*{thm*}{Theorem}
\theoremstyle{definition}

 \newcommand{\D}{\mathbb{D}}

 \newcommand{\ot}{\otimes}

\newcommand{\clb}{\mathcal{B}}

\newcommand{\cle}{\mathcal{E}}
\newcommand{\clf}{\mathcal{F}}
\newcommand{\clg}{\mathcal{G}}
\newcommand{\clh}{\mathcal{H}}
\newcommand{\clk}{\mathcal{K}}
\newcommand{\cll}{\mathcal{L}}
\newcommand{\clm}{\mathcal{M}}

\newcommand{\clo}{\mathcal{O}}

\newcommand{\clq}{\mathcal{Q}}

\newcommand{\cls}{\mathcal{S}}

\newcommand{\clw}{\mathcal{W}}

\newcommand{\Z}{\mathbb{Z}_+}
\newcommand{\C}{\mathbb{C}}
\newcommand{\z}{\bm{z}}
\newcommand{\w}{\bm{w}}
\newcommand{\raro}{\rightarrow}
\newcommand{\NI}{\noindent}
\newcommand{\T}{\tilde{T}}
\newcommand{\h}{\tilde{\mathcal{H}}}

\begin{document}
%\today

\title[Multiplicities, invariant subspaces and an additive formula]{Multiplicities, invariant subspaces and an additive formula}

\dedicatory{Dedicated to Professor Kalyan Bidhan Sinha on the
occasion of his 75th birthday}

\author[Chattopadhyay] {Arup Chattopadhyay}
\address{Department of Mathematics, Indian Institute of Technology Guwahati, Guwahati, 781039, India}
\email{arupchatt@iitg.ac.in, 2003arupchattopadhyay@gmail.com}

\author[Sarkar]{Jaydeb Sarkar}
\address{Indian Statistical Institute, Statistics and Mathematics Unit, 8th Mile, Mysore Road, Bangalore, 560059,
India}
\email{jay@isibang.ac.in, jaydeb@gmail.com}

\author[Sarkar]{Srijan Sarkar}
\address{Department of Mathematics, Indian Institute of Science, Bangalore, 560012, India}
\email{srijans@iisc.ac.in,
srijansarkar@gmail.com}



\begin{abstract}
Let $T = (T_1, \ldots, T_n)$ be a commuting tuple of bounded linear
operators on a Hilbert space $\mathcal{H}$. The multiplicity of $T$
is the cardinality of a minimal generating set with respect to $T$.
In this paper, we establish an additive formula for multiplicities
of a class of commuting tuples of operators. A special case of the
main result states the following: Let $n \geq 2$, and let
$\mathcal{Q}_i$, $i = 1, \ldots, n$, be a proper closed shift
co-invariant subspaces of the Dirichlet space or the Hardy space
over the unit disc in $\mathbb{C}$. If $\mathcal{Q}_i^{\bot}$, $i = 1, \ldots, n$, is a zero-based shift invariant subspace, then the multiplicity of the
joint $M_{\bm z} = (M_{z_1}, \ldots, M_{z_n})$-invariant subspace
$(\mathcal{Q}_1 \otimes \cdots \otimes \mathcal{Q}_n)^\perp$ of the
Dirichlet space or the Hardy space over the unit polydisc in
$\mathbb{C}^n$ is given by
\[
\mbox{mult}_{M_{\bm z}|_{ (\mathcal{Q}_1 \otimes \cdots \otimes
\mathcal{Q}_n)^\perp}} (\mathcal{Q}_1 \otimes \cdots \otimes
\mathcal{Q}_n)^\perp = \sum_{i=1}^n
(\mbox{mult}_{M_z|_{\mathcal{Q}_i^\perp}} (\mathcal{Q}_i^{\bot})) =
n.
\]
A similar result holds for the Bergman space over the unit polydisc.
\end{abstract}
\maketitle

\section{Introduction}

This paper is concerned with an additive formula for a numerical
invariant of commuting tuples of bounded linear operators on Hilbert
spaces. The additive formula arises naturally in connection with a class of simple invariant subspaces of the two-variable Hardy space $H^2(\D^2)$ \cite{CDS1}. From function Hilbert space point of view, our additive formula is more refined for zero-based invariant subspaces of the Dirichlet space, the Hardy space, the Bergman space and the weighted Bergman spaces over the open unit polydisc $\D^n$ in $\C^n$.

To be more specific, let us first define the numerical invariant.
Given an $n$-tuple of commuting bounded linear operators $T := (T_1,
\ldots, T_n)$ on a Hilbert space $\clh$, we denote by
\[
\mbox{mult}_T(\clh) = \min \{\# G : [G]_T = \clh, G \subseteq
\clh\},
\]
where
\[
[G]_T =\overline{\mbox{span}} \{T^{\bm{k}} (G): {\bm{k} \in \Z^n}\},
\]
and $T^{\bm{k}} = T_1^{k_1} \cdots T_n^{k_n}$ for all ${\bm{k}} =
(k_1, \ldots, k_n) \in \Z^n$. If
\[
\mbox{mult}_T(\clh) = m < \infty,
\]
then we say that the \textit{multiplicity} of $T$ is $m$. One also says that $T$ is $m$-cyclic. If $m = 1$, then we also say that $T$ is \textit{cyclic}, or simply \textit{cyclic}. A subset $G$ of $\clh$ is said to be
\textit{generating subset} with respect to $T$ if $[G]_T = \clh$.

We pause to note that the computation of multiplicities of (even concrete and simple) bounded linear operators is a challenging problem (perhaps due to its inherent dynamical nature). We refer Rudin \cite{R} for concrete (as well as pathological) examples of invariant subspaces of $H^2(\D^2)$ of infinite multiplicities and \cite{CDS1, CDS3, III, I31, I32} for some definite results on computations of multiplicities (also see
\cite{XF}).

The following example, as hinted above, illustrates the complexity
of computations of the multiplicities of general function Hilbert
spaces. As a first step, we consider the Hardy space $H^2(\D)$ over
$\D$ (the space of all square summable analytic functions on $\D$)
and the multiplication operator $M_z$ by the coordinate function
$z$. Let $\cls$ be a closed $M_z$-invariant subspace of $H^2(\D)$.
Then $\clq = \cls^\perp$ is a closed $M_z^*$-invariant subspace of
$H^2(\D)$. It then follows from Beurling that
\[
\mbox{mult}_{M_z|_{\cls}} (\cls) = 1,
\]
that is, $M_z|_{\cls}$ on $\cls$ is cyclic. Moreover, taking into
account that $\mbox{mult}_{M_z} (H^2(\D)) = 1$, we obtain (cf.
Proposition \ref{prop-Qrank})
\[
\mbox{mult}_{P_{\clq}M_z|_{\clq}} (\clq) = 1,
\]
where $P_{\clq}$ denote the orthogonal projection of $H^2(\D)$ onto
$\clq$.

\NI Now we consider the commuting pair of multiplication operators
$M_{\bm z} = (M_{z_1}, M_{z_2})$ on $H^2(\D^2)$ (the Hardy space over the
bidisc). Observe that $H^2(\D^2) \cong H^2(\D) \otimes H^2(\D)$. Let
$\clq_1$ and $\clq_2$ be two non-trivial closed $M_z^*$-invariant
subspaces of $H^2(\D)$. Then $\clq_1 \otimes \clq_2$ is a joint
$(M_{z_1}^*, M_{z_2}^*)$-invariant subspace of $H^2(\D^2)$, and so
$(\clq_1 \otimes \clq_2)^\perp$ is a joint $(M_{z_1},
M_{z_2})$-invariant subspace of $H^2(\D^2)$. Set $M_{\bm z}|_{(\clq_1
\otimes \clq_2)^\perp} = (M_{z_1}|_{(\clq_1 \otimes \clq_2)^\perp},
M_{z_2}|_{(\clq_1 \otimes \clq_2)^\perp})$. An equivalent
reformulation of Douglas and Yang's question (see page 220 in
\cite{DY} and also \cite{CDS1}) then takes the following form: Is
\[
\mbox{mult}_{M_{\bm z}|_{(\clq_1 \otimes \clq_2)^\perp}} (\clq_1 \otimes
\clq_2)^\perp = 2?
\]
The answer to this question is yes and was obtained by Das along
with the first two authors in \cite{CDS1}. This result immediately
motivates (see page 1186, \cite{CDS1}) the following natural
question: Consider the joint $M_{\bm z} = (M_{z_1}, \ldots, M_{z_n})$-invariant
subspace $(\clq_1 \otimes \cdots \otimes \clq_2)^\perp$ of
$H^2(\D^n)$ where $\clq_1, \ldots, \clq_n$ are non-trivial closed
$M_z^*$-invariant subspaces of $H^2(\D)$. Is then
\[
\mbox{mult}_{M_{\bm z}|_{(\clq_1 \otimes \cdots \otimes \clq_n)^\perp}}
(\clq_1 \otimes \cdots \otimes \clq_n)^\perp = n?
\]
This can be reformulated more concretely as follows: Let $\clh_i$ be the Dirichlet space, the Hardy space, the Bergman space, or the weighted Bergman spaces over $\D$ (or, more generally, a reproducing kernel Hilbert spaces of analytic functions on $\D$ for which the operator $M_z$ of multiplication by the coordinate function $z$ on $\clh_i$ is bounded), $i = 1, \ldots, n$. Suppose $\clq_i^{\bot}$ is an $M_z$-invariant closed subspace of $\clh_i$, $i = 1, \ldots, n$. Is then
\[
\mbox{mult}_{M_{\bm z}|_{(\clq_1 \otimes \cdots \otimes \clq_n)^\perp}}
(\clq_1 \otimes \cdots \otimes \clq_n)^\perp = \sum_{i=1}^n
(\mbox{mult}_{M_z|_{\clq_i^\perp}}~\big(\mathcal{Q}_i^{\bot}))?
\]
In this paper, we aim to propose an approach to verify the above equality for a large class of function Hilbert spaces over $\D^n$. The methods and techniques used in this paper are completely different from \cite{CDS1}, and can also be applied for proving more powerful results in the setting of
general Hilbert spaces. There is indeed a more substantial answer, valid in a larger context of tensor products of Hilbert spaces (see Theorem
\ref{thm-main1}).

Let $\clh \subseteq \clo(\D)$ be a reproducing kernel Hilbert space (or, the Dirichlet space, the Hardy space, the Bergman space, or the weighted Bergman spaces over $\D$) and let the operator $M_z$ is bounded on $\clh$. Suppose $\cls$ is a $M_z$-invariant closed subspace of $\clh$. We say that $\cls$ is a \textit{zero-based invariant subspace} if there exists $\lambda \in \D$ such that $f(\lambda) = 0$ for all $f \in \cls$.

A particular case of our main theorem is the following: Let $\clh_i$ be the
Dirichlet space, the Hardy space, the Bergman space, or the weighted Bergman spaces over $\D$. Let $\cls_i$ be an $M_z^*$-invariant closed subspace of $\clh_i$, $i = 1, \ldots, n$. Suppose $\cls_i := \clq_i^{\bot}$ is a zero-based $M_z$-invariant closed subspace of $\clh_i$ such that
\[
\mbox{dim} (\cls_i \ominus z \cls_i) < \infty \quad \mbox{and} \quad [\cls_i \ominus z \cls_i]_{M_{z}|_{\cls_i}} = \cls_i,
\]
for all $i = 1, \ldots, n$, then
\[
\mbox{mult}_{M_{\bm z}|_{(\clq_1 \otimes \cdots \otimes \clq_n)^\perp}}
(\clq_1 \otimes \cdots \otimes \clq_n)^\perp = \sum_{i=1}^n
(\mbox{mult}_{M_z|_{\clq_i^\perp}}~\big(\mathcal{Q}_i^{\bot})) = \sum_{i=1}^n \mbox{dim} (\cls_i \ominus z \cls_i).
\]
Note that the finite dimensional and generating subspace assumptions are automatically satisfied if $\clh_i$ is the Hardy space or the Dirichlet space. However, if $\cls$ is an $M_z$-invariant closed subspace of the Bergman space over $\D$, then 
\[
\mbox{dim} (\cls \ominus z \cls) \in \mathbb{N} \cup \{\infty\}.
\]
We refer the reader to \cite{ABFP, H1, H2} for more information. See also \cite{H3} for related results in the setting of weighted Bergman spaces over $\D$.

The proof of the above additivity formula uses generating
wandering subspace property, geometry of (tensor product) Hilbert
spaces and subspace approximation technique.

The paper is organized as follows. In Section 2, we set up notation
and prove some basic results on weak multiplicity of (not necessarily commuting) $n$-tuples of operators on Hilbert spaces. In Section 3, we study a lower bound multiplicity of joint invariant subspaces of a class of commuting $n$-tuples of operators. The main theorem on additivity formula is proved in Section 4. The paper is concluded in Section 5 with corollaries of the main theorem and some general discussions.

\newsection{Notation and basic results}\label{sec-prel}

In this section, we introduce the notion of weak multiplicities and describe some preparatory results. This notion is not absolutely needed for the main results of this paper as we shall mostly work in the setting of multiplicities. However, we believe that the idea of weak multiplicities of (not necessary commuting) tuples of operators might be of independent interest.

\textsf{Throughout this paper the following notation will be
adopted:} \textit{$T_i$ is a bounded linear operator on a separable
Hilbert space $\clh_i$, $i = 1, \ldots, n$, and $n \geq 2$. We set
\[
\h = \clh_1 \otimes \cdots \otimes \clh_n,
\]
and
\[
\T = (\T_1, \ldots, \T_n).
\]
where
\[
\T_i = I_{\clh_1} \otimes \cdots \otimes I_{\clh_{i-1}} \otimes
{T_i} \otimes I_{\clh_{i+1}} \otimes \cdots \otimes I_{\clh_n} \in
\clb(\tilde{\clh}),
\]
for all $i = 1, \ldots, n$.} It is now clear that $(\T_1, \ldots,
\T_n)$ is a doubly commuting tuple of operators on $\h$ (that is, $\T_i \T_j = \T_j \T_i$ and $\T_p^* \T_q = \T_q \T_p^*$ for all $1 \leq i, j \leq n$ and $1 \leq p < q \leq n$). Moreover, if
$\text{mult}_{T_i}(\clh_i) = 1$ for all $i = 1, \ldots, n$, then
$\text{mult}_{\T}(\h) = 1$. We denote by $\mathbb{D}^n$ the unit
polydisc in $\mathbb{C}^n$ and by $\bm{z}$ the element $(z_1,
\ldots, z_n)$ in $\mathbb{C}^n$.

The above notion of ``tensor product of operators'' is suggested by
natural (and analytic) examples of reproducing kernel Hilbert spaces
over product domains in $\mathbb{C}^n$. For instance, if
$\{\alpha_1, \ldots, \alpha_n\} \subseteq \mathbb{N}$, then
\[
K_{\bm{\alpha}}(\z, \w) : = \prod_{i=1}^n \frac{1}{(1 - z_i
\bar{w}_i)^{\alpha_i}} \quad \quad (\z, \w \in \D^n),
\]
is a positive definite kernel over the polydisc $\D^n$, and the
multiplication operator tuple $(M_{z_1}, \ldots, M_{z_n})$ defines
bounded linear operators on the corresponding reproducing kernel
Hilbert space $L^2_{\bm{\alpha}}(\D^n)$ (known as the weighted
Bergman space over $\mathbb{D}^n$ with weight $\bm{\alpha} =
(\alpha_1, \ldots, \alpha_n)$). It follows that (cf. \cite{T})
\[
\h = L^2_{\alpha_1}(\D) \otimes \cdots \otimes L^2_{\alpha_n}(\D),
\quad \text{and} \quad \tilde{M}_z = (\tilde{M}_{z_1}, \ldots,
\tilde{M}_{z_n}),
\]
where ${M}_{z_i}$ denotes the multiplication operator $M_z$ on
$L^2_{\alpha_i}(\D)$, $i = 1, \ldots, n$. In particular, if
$\bm{\alpha} = (1, \ldots, 1)$, then $\h = H^2(\D^n)$ is the well
known Hardy space over the unit polydisc. We also refer the reader
to Popescu \cite{P1, P2} for elegant and rich theory of ``tensor
product of operators'' in multivariable operator theory.

Let $\clh$ be a Hilbert space, and let $A = (A_1, \ldots, A_n)$ be
an $n$-tuple (not necessarily commuting) of bounded linear operators
on $\clh$. Let
\[
\mbox{w-mult}_A(\clh) = \min \{\# G : [G]_A = \clh, G \subseteq
\clh\},
\]
where
\[
[G]_A =\overline{\mbox{span}} \{A^{\bm{k}} (G): {\bm{k} \in \Z^n}\},
\]
and $A^{\bm{k}} = A_1^{k_1} \cdots A_n^{k_n}$ for all ${\bm{k}} \in
\Z^n$. If $\mbox{w-mult}_A(\clh) = m < \infty$, then we say that the \textit{weak multiplicity} of $A$ is $m$. We say that $A$ is \textit{weakly cyclic} if $\mbox{w-mult}_A(\clh) = 1$. A subset $G$ of $\clh$ is said to be \textit{weakly generating} with respect to $A$ if $[G]_A = \clh$.

Now let $\cll$ be a closed subspace of
$\clh$. Then
\[
\clw_{A}(\cll) := \cll \ominus \sum_{i=1}^n A_i \cll,
\]
is called the \textit{wandering subspace} of $\cll$ with respect to
$P_{\cll}A|_{\cll}$. If, in addition
\[
\cll = \bigvee_{\bm k \in \Z^n} (P_{\cll} A|_{\cll})^{\bm
k}(\clw_{A}(\cll)),
\]
then we say that $P_{\cll}A|_{\cll}$ satisfies the \textit{weakly generating
wandering subspace property}. Here $P_{\cll} A|_{\cll} = (P_{\cll} A_1|_{\cll}, \ldots, P_{\cll}
A_n|_{\cll})$ and
\[
(P_{\cll} A|_{\cll})^{\bm k} = (P_{\cll} A_1|_{\cll})^{k_1} \cdots
(P_{\cll} A_n|_{\cll})^{k_n},
\]
for all $\bm{k} \in \Z^n$.

Note that if $A$ is commuting and $\cll$ is joint $A$-invariant
subspace (that is, $A_i \cll \subseteq \cll$ for all $i = 1 \ldots,
n$), then weakly generating wandering subspace property is commonly
known as \textit{generating wandering subspace property}.

We now proceed to relate weak multiplicities and dimensions of
weakly generating wandering subspaces. Let $A$ be an $n$-tuple of
bounded linear operators on $\clh$, $\cll$ be a joint $A$-invariant
subspace of $\clh$, and let $\clm$ be a closed subspace of $\cll$. Then
\[
P_{\clw_{A}(\cll)}([\clm]_A) = P_{\clw_{A}(\cll)}(\clm),
\]
since
\[
P_{\clw_{A}(\cll)}(A^{\bm{k}} \clm) = 0 \quad \text{for all} \quad \bm{k} \in \Z^n\setminus \{0\}.
\]
Now suppose that $[\clm]_A = \cll$, that is, $\clm$ is a weakly generating subspace of $\cll$ with respect to $A$. Then
\[
\clw_A(\cll) = P_{\clw_{A}(\cll)}(\clm).
\]
Hence
\[
\text{w-mult}_{A|_{\cll}}(\cll) \geq \dim \clw_A(\cll).
\]
Moreover, if $\cll$ satisfies the weakly generating wandering
subspace property, then
\[
\mbox{w-mult}_{A}(\cll) = \dim \clw_{A}(\cll).
\]
Therefore we have proved the following:

\begin{propn}\label{prop-wwrank}
Let $\cll$ be a closed joint $A$-invariant subspace of $\clh$. If
$\cll$ satisfies the weakly generating wandering subspace property with respect to $A_{\cll}$,
then $\mbox{w-mult}_{A}(\cll) = \dim \clw_{A}(\cll)$.
\end{propn}

We now proceed to a variation of Lemma 2.1 in \cite{CDS1} which
relates the multiplicity of a commuting tuple of operators with the
weak-multiplicity of the compressed tuple to a semi-invariant
subspace.

\begin{lemma}\label{lemma-cds}
Let $A$ be an $n$-tuple of bounded linear operators on a Hilbert
space $\clh$. Let $\cll_1$ and $\cll_2$ be two joint $A$-invariant
subspaces of $\clh$ and $\cll_2 \subseteq \cll_1$. If $\cll = \cll_1
\ominus \cll_2$, then
\[
\mbox{w-mult}_{P_{\cll} A|_{\cll}}(\cll) \leq
\mbox{w-mult}_{A|_{\cll_1}}(\cll_1).
\]
\end{lemma}
\begin{proof}
We have $P_{\cll} A_j P_{\cll} = P_{\cll} A_j P_{\cll_1} - P_{\cll}
A_j P_{\cll_2}$ and thus by $A_j \cll_2 \subseteq \cll_2$ we infer
that
\[
P_{\cll} A_j P_{\cll} = P_{\cll} A_j P_{\cll_1},
\]
for all $j = 1, \ldots, n$. Since $A_j \cll_1 \subseteq \cll_1$, we
have
\[
(P_{\cll} A_i P_{\cll}) (P_{\cll} A_j P_{\cll}) = P_{\cll} A_i P_{\cll_1} A_j P_{\cll_1},
\]
that is
\[
(P_{\cll} A_i P_{\cll}) (P_{\cll} A_j P_{\cll}) = P_{\cll} (A_i
A_j)P_{\cll_1},
\]
for all $i, j = 1, \ldots, n$, and so
\[
(P_{\cll} A P_{\cll})^{\bm{k}} = P_{\cll} A^{\bm{k}} P_{\cll_1},
\]
for all $\bm{k} \in \Z^n$. Clearly, if $G$ is a minimal generating
subset of $\cll_1$ with respect to $A|_{\cll_1}$, then $P_{\cll} G$
is a generating subset of $\cll$ with respect to $P_{\cll}
A|_{\cll}$, and thus $\mbox{w-mult}_{P_{\cll} A|_{\cll}}(\cll) \leq
\mbox{w-mult}_{A|_{\cll_1}}(\cll_1)$. This completes the proof of the lemma.
\end{proof}

In particular, if $\cll_1 = \clh$, then $\clq := \clh \ominus
\cll_2$ is a joint $(A_1^*, \ldots, A_n^*)$-invariant subspace of $\clh$. In this case, denote by $C_i = P_{\clq} A_i|_{\clq}$ the compression of $A_i$, $i = 1, \ldots, n$, and define the $n$-tuple on $\clq$ as
\[
C_{\clq} = (C_1, \ldots, C_n).
\]
Then we have the following estimate:
\[
\mbox{w-mult}_{C_{\clq}}(\clq) \leq \text{w-mult}_{A}(\clh).
\]
Moreover, we also have

\begin{cor}\label{prop-Qrank}
Let $A = (A_1, \ldots, A_n)$ be a commuting tuple of bounded linear operators on a Hilbert space
$\clh$. If $\clq$ is a closed joint $A^*$-invariant subspace of
$\clh$, then
\[
\mbox{mult}_{C_{\clq}}(\clq) \leq \text{mult}_{A}(\clh).
\]
\end{cor}

This has the following immediate (and well-known) application: Suppose $A$ is a commuting tuple on $\clh$. If $A$ is cyclic, then $C_{\clq}$ on $\clq$ is also cyclic.

\section{A lower bound for multiplicities}\label{sec-lowerbound}

In this section, we first lay out the setting of joint invariant subspaces of our discussions throughout the paper. Then we present a lower bound of multiplicities of those joint invariant subspaces. We begin by recalling the following useful lemma (cf. Lemma 2.5, \cite{sarkar}):

\begin{lemma}\label{lemma-old}
If $\{A_i\}_{i=1}^n$ is a commuting set of orthogonal projections on
a Hilbert space $\clk$, then $\displaystyle \cll = \sum_{i=1}^n
\mbox{ran} A_i$ is a closed subspace of $\clk$, and
\[
\begin{split}
P_{\cll} & = I - \prod_{i=1}^n (I - A_i)
\\
& = A_1(I-A_{2}) \ldots (I-A_{n}) \oplus A_{2}(I-A_{3}) \ldots
(I-A_{n}) \oplus \ldots + A_{n-1}(I-A_{n}) \oplus A_n.
\end{split}
\]
\end{lemma}

Next, we introduce the invariant subspaces of interest. Again, we
continue to follow the notations as introduced in Section
\ref{sec-prel}.

Let $\clh_i$ be a Hilbert space, $T_i$ a bounded linear operator on
$\clh_i$, and let $\clq_i$ be a closed $T_i^*$-invariant subspace of
$\clh_i$, $i = 1, \ldots, n$. Set $\cls_i = \clq_i^\perp$ and
\[
\quad P_i = P_{\cls_i} \quad\mbox{and} \quad Q_i = I_{\clh_i} - P_{\cls_i},
\]
for all $i = 1, \ldots, n$. Recall again that
\[
\tilde{P}_i = I_{\clh_1} \ot \ldots \ot I_{\clh_{i-1}} \ot
P_{\cls_{i}} \ot I_{\clh_{i+1}} \ot  \ldots \ot I_{\clh_n} \in
\clb({\h}),
\]
and
\[
\tilde{P}_i \tilde{P}_j = \tilde{P}_j \tilde{P}_i,
\]
for all $i, j = 1, \ldots, n$. By Lemma \ref{lemma-old}, it then follows that
\begin{equation}\label{eq-S Pi}
\cls = \sum_{i=1}^n \mbox{ran} \tilde{P}_i,
\end{equation}
is a joint $\tilde{T}$-invariant subspace of ${\h}$. Moreover
\[
\cls = (\clq_1 \otimes \cdots \otimes \clq_n)^\perp.
\]
Our main goal is to compute the multiplicity of the commuting tuple $\tilde{T}|_{\cls} = (\tilde{T}_1|_{\cls}, \ldots, \tilde{T}_n|_{\cls})$ on $\cls$.

For each $i = 1, \ldots, n$, define $X_i \in \clb(\tilde{\clh})$ by
\[
X_i = \tilde{P}_i \tilde{Q}_{i+1} \ldots \tilde{Q}_{n}.
\]
Then $X_i^2 = X_i = X_i^*$ and
\[
X_p X_q = 0,
\]
for all $i = 1, \ldots, n$, and $p \neq q$. This implies that $\{X_i\}_{i=1}^n$ is a set of orthogonal projections with orthogonal ranges. Then, by virtue of \eqref{eq-S Pi} one can further rewrite $\cls$ as
\begin{equation}\label{eq-S Pi Xi}
\cls = \sum_{i=1}^n \mbox{ran} \tilde{P}_i = \bigoplus_{i=1}^n
\mbox{ran} X_i,
\end{equation}
and by Lemma \ref{lemma-old} one represent $P_{\cls}$ as
\[
P_{\cls} =  \bigoplus_{i=1}^n X_i.
\]
Define
\begin{equation}\label{eq-F}
\clf = \mbox{ran} X_1 \oplus \mbox{ran}(\tilde{Q}_1 X_2) \oplus
\cdots \oplus \mbox{ran}(\tilde{Q}_1 \cdots \tilde{Q}_{n-1}X_n).
\end{equation}\
Then, as easily seen
\[
\tilde{Q}_i X_j = X_j \tilde{Q}_i,
\]
for all $1 \leq i \leq j$ and $j = 1, \ldots, n$, it follows that
\[
\mbox{ran} (\tilde{Q}_1 \cdots \tilde{Q}_p X_{p+1}) \subseteq
\mbox{ran} X_{p+1},
\]
for all $p = 1, \ldots, n-1$, and consequently
\[
\cls \supseteq \clf .
\]
Our first aim is to analyze the closed subspace $\clf$ and to
construct $n-1$ nested (and suitable) closed subspaces
$\{\clf_i\}_{i=1}^{n-1}$ such that
\[
\cls \supseteq \clf_1 \supseteq \cdots \supseteq \clf_{n-1} = \clf.
\]
To this end, first set
\[
\clf_1 = \mbox{ran} X_1 \oplus \mbox{ran} X_2 \oplus \cdots \oplus
\mbox{ran}X_{n-1} \oplus \mbox{ran} (\tilde{Q}_{n-1} X_n),
\]
and define
\[
\clf_2 = \mbox{ran} X_1 \oplus \mbox{ran} (\tilde{Q}_1 X_2) \oplus
\cdots \oplus \mbox{ran}(\tilde{Q}_1 X_{n-1}) \oplus \mbox{ran}
(\tilde{Q}_1 \tilde{Q}_{n-1} X_n).
\]
We then proceed to define $\clf_i$, $i = 2, \ldots, n-1$, as
\begin{equation}\label{eq-Fi}
\clf_i = \mbox{ran} \Big(X_1 \oplus  \tilde{Q}_1 X_2 \oplus \cdots
\oplus (\prod_{t=1}^{i-1} \tilde{Q}_t) X_i \oplus \cdots \oplus
(\prod_{t=1}^{i-1} \tilde{Q}_t) X_{n-1} \oplus (\prod_{t=1}^{i-1}
\tilde{Q}_t \tilde{Q}_{n-1}) X_n\Big).
\end{equation}
Therefore
\begin{equation}\label{eq-Fi}
P_{\clf_i} = X_1 \oplus  \tilde{Q}_1 X_2 \oplus \cdots
\oplus (\prod_{t=1}^{i-1} \tilde{Q}_t) X_i \oplus \cdots \oplus
(\prod_{t=1}^{i-1} \tilde{Q}_t) X_{n-1} \oplus (\prod_{t=1}^{i-1}
\tilde{Q}_t \tilde{Q}_{n-1}) X_n,
\end{equation}
for all $i = 2, \ldots, n-1$. Therefore, denoting
\[
A = (\prod_{t=1}^{i-2} \tilde{Q}_t)\tilde{P}_{i-1},
\]
we have
\begin{equation}\label{eq-Fii-1}
P_{\clf_{i-1} \ominus \clf_i} = A (X_i \oplus X_{i+1} \oplus \cdots
\oplus X_{n-1} \oplus \tilde{Q}_{n-1} X_n),
\end{equation}
for all $i = 2, \ldots, n-1$. Since $A X_p = X_p A$ for all $p = i,
\ldots, n$, the above formula yields
\[
P_{\clf_{i-1} \ominus \clf_i} = (X_i \oplus X_{i+1} \oplus \cdots
\oplus X_{n-1} \oplus \tilde{Q}_{n-1} X_n) A.
\]

Let $i \in \{2, \ldots, n-1\}$ be a fixed natural number. \textsf{We
claim that} $\clf_{i-1} \ominus \clf_i$ is a joint $P_{\clf_{i-1}}
\tilde{T} P_{\clf_{i-1}}$-invariant subspace, that is
\[
P_{\clf_{i-1}} \tilde{T}_j (\clf_{i-1} \ominus \clf_i) \subseteq
\clf_{i-1} \ominus \clf_i.
\]
or, equivalently
\[
(P_{\clf_{i-1}} \tilde{T}_j P_{\clf_{i-1}}) P_{\clf_{i-1} \ominus
\clf_i} = P_{\clf_{i-1} \ominus \clf_i} \tilde{T}_j|_{\clf_{i-1}
\ominus \clf_i},
\]
for all $j = 1, \ldots, n$. There are four cases:

\vspace{0.1in}

\NI\textsf{Case I:} If $j > i$, then one has $\tilde{T}_j A = A
\tilde{T}_j$ and so
\[
P_{\clf_{i-1} \ominus \clf_i} \tilde{T}_{j} P_{\clf_{i-1} \ominus
\clf_i} = A (X_i \oplus X_{i+1} \oplus \cdots \oplus \tilde{Q}_{n-1}
X_n) \tilde{T}_{j}(X_i \oplus X_{i+1} \oplus \cdots \oplus
\tilde{Q}_{n-1} X_n).
\]
On the other hand, since
\[
P_{\clf_{i-1}} \tilde{T}_{j} P_{\clf_{i-1} \ominus \clf_i} =
P_{\clf_{i-1}} A \tilde{T}_j (X_i \oplus \cdots \oplus X_{j} \oplus
\cdots \oplus \tilde{Q}_{n-1} X_n),
\]
and
\[
\begin{split}
P_{\clf_{i-1}} = X_1 \oplus & (\tilde{Q}_1 X_2) \oplus \cdots \oplus
(\prod_{t=1}^{i-2} \tilde{Q}_t X_{i-1}) \oplus (\prod_{t=1}^{i-2}
\tilde{Q}_t X_{i}) \oplus
\\
& \cdots \oplus (\prod_{t=1}^{i-2}
\tilde{Q}_t X_{n-1}) \oplus (\prod_{t=1}^{i-2} \tilde{Q}_t
\tilde{Q}_{n-1} X_n),
\end{split}
\]
it follows that
\[
P_{\clf_{i-1}} \tilde{T}_{j} P_{\clf_{i-1} \ominus \clf_i} =
A(X_{i-1} \oplus X_i \oplus \cdots \oplus \tilde{Q}_{n-1} X_n)
\tilde{T}_j (X_i \oplus \cdots \oplus \tilde{Q}_{n-1} X_n),
\]
as $X_t A = 0$ for all $t = 1, \ldots, i-2$, and $\displaystyle
\prod_{t=1}^{i-2} \tilde{Q}_t A = A$. Moreover, since
\[
X_{i-1} \tilde{T}_j = (\tilde{P}_{i-1} \tilde{Q}_i \cdots
\tilde{Q_j} \cdots \tilde{Q}_n) \tilde{T}_j = \tilde{P}_{i-1}
\tilde{Q}_i \cdots \widetilde{Q_j T_j Q_j} \cdots \tilde{Q}_n,
\]
it follows that
\[
X_{i-1} \tilde{T}_j X_t = 0,
\]
for all $t = i, \ldots, n$. This leads to
\[
P_{\clf_{i-1}} \tilde{T}_{j} P_{\clf_{i-1} \ominus \clf_i} =  A (X_i
\oplus X_{i+1} \oplus \cdots \oplus \tilde{Q}_{n-1} X_n)
\tilde{T}_{j}(X_i \oplus X_{i+1} \oplus \cdots \oplus
\tilde{Q}_{n-1} X_n).
\]

\NI\textsf{Case II:} If $j = i$, then
\[
\tilde{T}_{i} P_{\clf_{i-1} \ominus \clf_i} = A ((\widetilde{T_i
P_i} \tilde{Q}_{i+1} \cdots \tilde{Q}_n) \oplus \tilde{T}_i X_{i+1}
\oplus \cdots \oplus \tilde{T}_i \tilde{Q}_{n-1} X_n),
\]
implies that
\[
\begin{split}
P_{\clf_{i-1}} \tilde{T}_{i} P_{\clf_{i-1} \ominus \clf_i} & =
(\prod_{t=1}^{i-2} \tilde{Q}_t) (X_{i-1} \oplus X_{i} \oplus \cdots
\oplus \tilde{Q}_{n-1} X_n) \tilde{T}_{i} P_{\clf_{i-1} \ominus
\clf_i}
\\
& =  (\prod_{t=1}^{i-2} \tilde{Q}_t) (X_{i-1} \oplus X_{i} \oplus
\cdots \oplus \tilde{Q}_{n-1} X_n) \tilde{T}_{i} A (X_i \oplus
X_{i+1} \oplus \cdots \oplus \tilde{Q}_{n-1} X_n)
\\
& =  A (X_{i-1} \oplus X_{i} \oplus \cdots \oplus \tilde{Q}_{n-1}
X_n) \tilde{T}_{i}(X_i \oplus X_{i+1} \oplus \cdots \oplus
\tilde{Q}_{n-1} X_n)
\\
& = P_{\clf_{i-1} \ominus \clf_i} \tilde{T}_i P_{\clf_{i-1} \ominus
\clf_i},
\end{split}
\]
where the next-to-last equality follows from the fact again that $A
\tilde{T}_i = \tilde{T}_i A$, $\displaystyle(\prod_{t=1}^{i-2}
\tilde{Q}_t) A = A$ and $X_{i-1} \tilde{T}_i X_t = 0$ for all $t =
i, \ldots, n$.

\NI\textsf{Case III:} Let $j = i-1$. Since
\[
\tilde{T}_{i-1} A = (\prod_{t=1}^{i-2} \tilde{Q}_t)
\widetilde{T_{i-1} P_{i-1}} = A \widetilde{T_{i-1} P_{i-1}},
\]
by setting
\[
\hat{A} = (\prod_{t=1}^{i-2} \tilde{Q}_t)\widetilde{T_{i-1}P_{i-1}},
\]
it follows that
\[
\tilde{T}_{i-1} P_{\clf_{i-1} \ominus \clf_i} = \hat{A} X_i \oplus
\hat{A} X_{i+1} \oplus \cdots \oplus \hat{A} X_{n-1} \oplus \hat{A}
\tilde{Q}_{n-1} X_n.
\]
Then $X_p \hat{A} = \hat{A} X_p$ for all $p = i, \ldots, n$, and $A
\hat{A} = \hat{A}$ implies that
\[
\begin{split}
P_{\clf_{i-1}} \tilde{T}_{i-1} P_{\clf_{i-1} \ominus \clf_i} & =
\hat{A} (X_i \oplus X_{i+1} \oplus \cdots \oplus X_{n-1} \oplus
\tilde{Q}_{n-1} X_n)
\\
& = P_{\clf_{i-1} \ominus \clf_i} \tilde{T}_{i-1} P_{\clf_{i-1}
\ominus \clf_i},
\end{split}
\]
where the second equality follows from \eqref{eq-Fii-1} and the fact
that $T_{i-1}P_{i-1} = P_{i-1}T_{i-1}P_{i-1}$.

\NI\textsf{Case IV:} Let $j < i-1$. Then it is clear that
\[
\tilde{T}_j P_{\clf_{i-1} \ominus \clf_i} = \hat{A} (X_i \oplus
X_{i+1} \oplus \cdots \oplus X_{n-1} \oplus \tilde{Q}_{n-1} X_n),
\]
where $\hat{A} = \tilde{T}_j {A}$, that is
\[
\hat{A} = \tilde{Q}_1 \cdots \tilde{Q}_{j-1} \widetilde{T_j Q_j}
\tilde{Q}_{j+1} \cdots \tilde{Q}_{i-2} \tilde{P}_{i-1}.
\]
Note that $X_t \hat{A} = \hat{A} X_t$ for all $t = i, \ldots, n$,
and
\[
A \hat{A} = \tilde{Q}_1 \cdots \tilde{Q}_{j-1} \widetilde{Q_j T_j
Q_j} \tilde{Q}_{j+1} \cdots \tilde{Q}_{i-2} \tilde{P}_{i-1}.
\]
Since $X_p X_q = \delta_{pq} X_p$ for all $p$ and $q$, it follows
that
\[
P_{\clf_{i-1}} \tilde{T}_j P_{\clf_{i-1} \ominus \clf_i} = A \hat{A}
(X_i \oplus X_{i+1} \oplus \cdots \oplus X_{n-1} \oplus
\tilde{Q}_{n-1} X_n).
\]
On the other hand, the representation of $\tilde{T}_j P_{\clf_{i-1}
\ominus \clf_i}$ above and \eqref{eq-Fii-1} yields
\[
P_{\clf_{i-1} \ominus \clf_i} \tilde{T}_j P_{\clf_{i-1} \ominus
\clf_i} = A \hat{A} (X_i \oplus X_{i+1} \oplus \cdots \oplus X_{n-1}
\oplus \tilde{Q}_{n-1} X_n),
\]
and proves the claim.

We turn now to prove that $(P_{\clf_i} \tilde{T}_1|_{\clf_i},
\ldots, P_{\clf_i} \tilde{T}_n|_{\clf_i})$ is a commuting tuple for
all $i = 1, \ldots, n-1$, that is
\[
P_{\clf_i} \tilde{T}_s P_{\clf_i} \tilde{T}_t P_{\clf_i} =
P_{\clf_i} \tilde{T}_t P_{\clf_i} \tilde{T}_s P_{\clf_i},
\]
for all $s, t = 1, \ldots, n$. Fix an $i \in \{1, \ldots, n-1\}$ and
let
\begin{equation}\label{eq-PFiMi}
P_{\clf_i} = M_1 \oplus \ldots \oplus M_n,
\end{equation}
where $M_j$, $j = 1, \ldots, n$, denotes the $j$-th summand in the
representation of $P_{\clf_i}$ in \eqref{eq-Fi}. Recalling the terms
in \eqref{eq-Fi}, we see that $M_j$ is a product of $n$ distinct
commuting orthogonal projections of the form $\tilde{P}_k$,
$\tilde{Q}_l$ and $\tilde{I}_{\clh_m}$, $1 \leq k,l,m \leq n$. For
each $s = 1, \ldots, n$, we set
\[
M_j = M_{j,s} \hat{M}_{j,s},
\]
where $M_{j,s}$ is the $s$-th factor of $M_j$ and $\hat{M}_{j,s}$ is
the product of the same factors of $M_j$, except the $s$-th factor
of $M_j$ is replaced by $\tilde{I}_{\clh_s}$. Note again that
$M_{j,s} = \tilde{P}_s, \tilde{Q}_s$, or $\tilde{I}_{\clh_s}$. We
first claim that
\begin{equation}\label{eq-MiTsMk =0}
M_j \tilde{T}_s M_k = 0,
\end{equation}
for all $j \neq k$. Indeed, if $M_{j,s} = \tilde{Q}_s$, then $M_j
\tilde{T}_s M_k = M_{j,s} \hat{M}_{j,s} \tilde{T}_s M_k$ yields
\[
\begin{split}
M_j \tilde{T}_s M_k = {M}_{j,s} \tilde{T}_s \hat{M}_{j,s} {M}_k =
{M}_{j,s} \tilde{T}_s {M}_{j,s} \hat{M}_{j,s} {M}_k = {M}_{j,s}
\tilde{T}_s {M}_{j} M_k = 0,
\end{split}
\]
as $\tilde{Q}_s \tilde{T}_s \tilde{Q}_s = \tilde{Q}_s \tilde{T}_s$.
Similarly, if $M_{j,s} = \tilde{P}_s$, then
\[
\begin{split}
M_j \tilde{T}_s M_k = {M}_{j} \hat{M}_{k,s} \tilde{T}_s M_{k,s} =
{M}_{j} \hat{M}_{k,s} M_{k,s} \tilde{T}_s M_{k,s} = {M}_j M_k
\tilde{T}_s M_{k,s} = 0,
\end{split}
\]
as $\tilde{P}_s \tilde{T}_s \tilde{P}_s = \tilde{T}_s \tilde{P}_s$.
The remaining case, $M_{j,s} = \tilde{I}_{\clh_s}$, follows from the
fact that
\[
M_j \tilde{T}_s M_k = \tilde{T}_s M_j M_k.
\]
This proves the claim. Hence the representation of $P_{\clf_i} \tilde{T}_s
P_{\clf_i}$ simplifies as
\begin{equation}\label{PFitildeTsPFi}
P_{\clf_i} \tilde{T}_s P_{\clf_i} = M_1 \tilde{T}_s M_1 \oplus
\cdots \oplus M_n \tilde{T}_s M_n.
\end{equation}
Thus
\[
P_{\clf_i} \tilde{T}_s P_{\clf_i} \tilde{T}_t P_{\clf_i} = M_1
\tilde{T}_s M_1 \tilde{T}_t M_1 \oplus \cdots \oplus M_n \tilde{T}_s
M_n \tilde{T}_t M_n.
\]
Now if $s \neq t$, then for each $j = 1, \ldots, n$, we have
\[
\begin{split}
M_j \tilde{T}_s M_j \tilde{T}_t M_j & = M_j \hat{M}_{j,s}
\tilde{T}_s M_{j,s} M_{j,t} \tilde{T}_t \hat{M}_{j,t} M_j
\\
& = (M_j \hat{M}_{j,s} M_{j,t}) \tilde{T}_s \tilde{T}_t (M_{j,s}
\hat{M}_{j,t} M_j)
\\
& = M_j \tilde{T}_s \tilde{T}_t M_j,
\end{split}
\]
and hence
\[
(P_{\clf_i} \tilde{T}_s P_{\clf_i}) (P_{\clf_i} \tilde{T}_t
P_{\clf_i}) = M_1 \tilde{T}_s \tilde{T}_t M_1 \oplus \cdots \oplus
M_n \tilde{T}_s \tilde{T}_t M_n.
\]
This completes the proof of the commutativity property of the tuple
$(P_{\clf_i} \tilde{T}_1|_{\clf_i}, \ldots, P_{\clf_i}
\tilde{T}_n|_{\clf_i})$, $i = 1, \ldots, n-1$. Furthermore, if $s=t$, then
\[
(M_j \tilde{T}_s M_j)^2 = M_j \tilde{T}_s^2 M_j.
\]
Indeed, if $M_{j,s} = \tilde{Q}_s$, then $M_j \tilde{T}_s M_j = M_j \tilde{T}_s \hat{M}_{j,s}$ gives us
\[
\begin{split}
M_j \tilde{T}_s M_j \tilde{T}_s M_j  = M_j \tilde{T}_s \hat{M}_{j,s} \tilde{T}_s \hat{M}_{j,s} M_j =  M_j \tilde{T}_s \tilde{T}_s \hat{M}_{j,s} M_j = M_j \tilde{T}_s^2 M_j.
\end{split}
\]
Similarly, if $M_{j,s} = \tilde{P}_s$ or $\tilde{I}_{\clh_s}$, then $M_j \tilde{T}_s M_j = \tilde{T}_s M_{j}$, and hence 
\[
M_j \tilde{T}_s M_j \tilde{T}_s M_j = M_j \tilde{T}_s^2 M_j.
\]
Hence we obtain 
\begin{equation}\label{powerpreserve}
(P_{\clf_i} \tilde{T}_s P_{\clf_i}) (P_{\clf_i} \tilde{T}_t
P_{\clf_i}) = M_1 \tilde{T}_s \tilde{T}_t M_1 \oplus \cdots \oplus
M_n \tilde{T}_s \tilde{T}_t M_n,
\end{equation}
for all $s, t = 1, \ldots, n$. 

Therefore, with the notations introduced above, we have proved the
following:

\begin{lemma}\label{lemma-S}
If $\cls = (\clq_1 \otimes \cdots \otimes \clq_n)^\perp$, then
$\cls$ is a joint $\tilde{T}$-invariant subspace of $\tilde{\clh}$
and
\[
\cls \supseteq \clf_1 \supseteq \cdots \supseteq \clf_{n-1} = \clf,
\]
where $\clf$ and $\clf_i$ are defined as in \eqref{eq-F} and
\eqref{eq-Fi}, respectively. Moreover
\[
P_{\clf_{i-1}} \tilde{T}|_{\clf_{i-1}} = (P_{\clf_{i-1}}
\tilde{T}_1|_{\clf_{i-1}}, \ldots, P_{\clf_{i-1}}
\tilde{T}_n|_{\clf_{i-1}}),
\]
is a commuting tuple and
\[
\Big(P_{\clf_{i-1}} \tilde{T}_j|_{\clf_{i-1}}\Big) (\clf_{i-1}
\ominus \clf_i) \subseteq \clf_{i-1} \ominus \clf_i,
\]
for all $i = 2, \ldots, n-1$, and $j = 1, \ldots, n$.
\end{lemma}

We now proceed to estimate a lower bound of $\mbox{mult}_{\tilde{T}|_{\cls}}(\cls)$. Note first that
$\mbox{ran}(\tilde{P}_{n-1} \tilde{P}_n)$ is a joint
$\tilde{T}$-invariant subspace and
\[
\clf_1 = \cls \ominus \mbox{ran}(\tilde{P}_{n-1} \tilde{P}_n).
\]
Then $\clf_1$ is a $\tilde{T}$-semi invariant subspace, which, by
Lemma \ref{lemma-cds}, implies that
\[
\mbox{mult}_{\tilde{T}|_{\cls}}(\cls) \geq
\mbox{mult}_{P_{\clf_1}\tilde{T}|_{\clf_1}}(\clf_1).
\]
Now consider the commuting $n$-tuple $P_{\clf_{1}}
\tilde{T}|_{\clf_{1}} = (P_{\clf_{1}} \tilde{T}_1|_{\clf_{1}},
\ldots, P_{\clf_{1}} \tilde{T}_n|_{\clf_{1}})$ on $\clf_1$. Then by
Lemma \ref{lemma-S} we infer that $\clf_1 \ominus \clf_2$ is a joint
$P_{\clf_{1}} \tilde{T}|_{\clf_1}$-invariant subspace of $\clf_1$.
But since $\clf_2 = \clf_1 \ominus (\clf_1 \ominus \clf_2)$, it
follows again by Lemma \ref{lemma-cds} that
\[
\mbox{mult}_{P_{\clf_1}\tilde{T}|_{\clf_1}}(\clf_1) \geq
\mbox{mult}_{P_{\clf_2}\tilde{T}|_{\clf_2}}(\clf_2).
\]
In general, by virtue of Lemma \ref{lemma-S}, we have
\[
\mbox{mult}_{P_{\clf_{i-1}}\tilde{T}|_{\clf_{i-1}}}(\clf_{i-1}) \geq
\mbox{mult}_{P_{\clf_i}\tilde{T}|_{\clf_i}}(\clf_i),
\]
for all $i = 2, \ldots, n-1$, and hence
\[
\mbox{mult}_{\tilde{T}|_{\cls}}(\cls) \geq
\mbox{mult}_{P_{\clf_1}\tilde{T}|_{\clf_1}}(\clf_1) \geq \ldots \geq
\mbox{mult}_{P_{\clf_{n-1}}\tilde{T}|_{\clf_{n-1}}}(\clf_{n-1}) =
\mbox{mult}_{P_{\clf} \tilde{T}|_{\clf}}(\clf),
\]
where (see (\ref{eq-F}))
\[
\clf = \mbox{ran} X_1 \oplus \mbox{ran}(\tilde{Q}_1 X_2) \oplus
\cdots \oplus \mbox{ran}(\tilde{Q}_1 \cdots \tilde{Q}_{n-1}X_n),
\]
and $X_i = \tilde{P}_i \tilde{Q}_{i+1} \cdots \tilde{Q}_n$, $i = 1,
\ldots, n$. We summarize the above discussion in the following
theorem:

\begin{thm}\label{thm-rank inequality}
Let $T_1, \ldots, T_n$ be bounded linear operators on Hilbert
spaces $\clh_1, \ldots, \clh_n$, respectively. If $\clq_i$ is a
$T_i^*$-invariant closed subspace of $\clh_i$, $i = 1, \ldots, n$,
and
\[
\cls = (\clq_1 \otimes \cdots \otimes \clq_n)^\perp,
\]
then
\[
\mbox{mult}_{\tilde{T}|_{\cls}}(\cls) \geq \mbox{mult}_{P_{\clf}
\tilde{T}|_{\clf}}(\clf).
\]
\end{thm}

\section{Additivity of multiplicities}

We now proceed to prove the reverse inequality in Theorem \ref{thm-rank inequality}. We start with a simple but useful lemma.

\begin{lemma}\label{lem-spec mult}
Let $(A_1,\ldots, A_n)$ be an $n$-tuple of bounded linear operators on a Hilbert space $\clh$. If $G$ is a subset of $\clh$ and $(\lambda_1,\ldots,\lambda_n) \in \mathbb{C}^n$, then
\[
[G]_{(A_1,\ldots, A_n)} = [G]_{(A_1 - \lambda_1 I_{\clh},\ldots, A_n - \lambda_n I_{\clh})}.
\]
\end{lemma}
\begin{proof}
Note that, given $p\in \mathbb{C}[z_1,\ldots,z_n]$ there exists $q \in \mathbb{C}[z_1,\ldots,z_n]$ such that
\begin{align*}
p(A_1,\ldots, A_n) &= p((A_1 - \lambda_1 I_{\clh} + \lambda_1 I_{\clh}),\ldots, (A_n - \lambda_n I_{\clh} + \lambda_n I_{\clh}))\\
&= q((A_1 - \lambda_1 I_{\clh}),\ldots, (A_n - \lambda_n I_{\clh})),
\end{align*}
which implies that
\[
[G]_{(A_1,\ldots, A_n)} \subseteq [G]_{(A_1 - \lambda_1 I_{\clh},\ldots, A_n - \lambda_n I_{\clh})}.
\]
The reverse inclusion follows similarly, and hence the result follows.
\end{proof}

Now we return to the problem of rank computation of $\cls$ as in Theorem \ref{thm-rank inequality}. \textsf{From now on, we will
use the setting and notations introduced in Section \ref{sec-lowerbound}.} Observe that, by \eqref{eq-F}, we have
\[
\clf = \clm_1 \oplus \cdots \oplus \clm_n,
\]
where
\[
\clm_i = \mbox{ran} \Big(\tilde{P}_i \prod_{j\neq i} \tilde{Q}_j \Big).
\]
By defining $M_i = P_{\clm_i}$, $i = 1, \ldots, n$, one has (see \eqref{eq-PFiMi})
\[
P_{\clf} = M_1 \oplus \cdots \oplus M_n.
\]
Recall, by virtue of \eqref{PFitildeTsPFi}, that
\begin{equation}\label{eq-PFTF = M}
P_{\clf} \tilde{T}_s P_{\clf} = M_1 \tilde{T}_s M_1 \oplus \cdots \oplus M_n \tilde{T}_s M_n,
\end{equation}
for all $s = 1, \ldots, n$. And, finally, recall that, by Lemma \ref{lemma-S}, $P_{\clf} \tilde{T} P_{\clf}$ is a commuting tuple on $\clf$. The equality in \eqref{eq-PFTF = M} implies that
\[
(P_{\clf} \tilde{T}_s P_{\clf}) \clm_i \subseteq \clm_i \quad \quad (s=1, \ldots, n),
\]
that is, $\clm_i$ is a joint $P_{\clf} \tilde{T} P_{\clf}$-invariant subspace of $\clf$ for all $i = 1, \ldots, n$. Then by virtue of \eqref{powerpreserve}, we have 
\[
(P_{\clf} \tilde{T}|_{\clf})^{\bm{k}} = \bigoplus_{i=1}^n P_{\clm_i} \tilde{T}^{\bm{k}}|_{\clm_i} \quad \quad (\bm{k} \in \mathbb{Z}_+^n).
\]
Now let $G$ be a minimal generating subset of $\clf$ with respect to $P_{\clf} \tilde{T}|_{\clf}$. Then
\[
\begin{split}
\clf = \overline{\mbox{span}} \{(P_{\clf} \tilde{T}|_{\clf})^{\bm{k}}(G): \bm{k} \in \mathbb{Z}_+^n\} \subseteq \bigoplus_{i=1}^n \Big(\overline{\mbox{span}} \{P_{\clm_i} \tilde{T}^{\bm{k}}|_{\clm_i}(G): \bm{k} \in \mathbb{Z}_+^n\}\Big) \subseteq \clf,
\end{split}
\]
and so
\[
\clf = \bigoplus_{i=1}^n \Big(\overline{\mbox{span}} \{P_{\clm_i} \tilde{T}^{\bm{k}}|_{\clm_i}(G): \bm{k} \in \mathbb{Z}_+^n\}\Big).
\]
\textsf{Now assume} that the point spectrum $\sigma_p(T_i^*|_{\clq_i}) \neq \emptyset$, $T_i|_{\cls_i}$ satisfies the generating wandering subspace property, and
\[
\mbox{dim} (\cls_i \ominus T_i \cls_i) < \infty,
\]
for all $i = 1, \ldots, n$. If we then let $\bar{\alpha_i} \in \sigma_p(T_i^*|_{\clq_i})$ and $T_i^* v_i = \bar{\alpha}_i v_i$ for some non-zero $v_i \in \clq_i$, then
\[
\cle_i := \mbox{ran} \Big(\tilde{P}_{\cls_i \ominus T_i \cls_i} \prod_{j \neq i} \tilde{P}_{\mathbb{C} v_j} \Big) \subseteq \clm_i,
\]
and
\[
\mbox{dim} \cle_i = \mbox{dim} (\cls_i \ominus T_i \cls_i) = \mbox{mult}_{T_i|_{\cls_i}}(\cls_i),
\]
for all $i = 1, \ldots, n$. Thus, if we set
\[
\cle = \cle_1 \oplus \cdots \oplus \cle_n,
\]
then $\cle \subseteq \clf$ and
\[
\mbox{dim} \cle = \sum_{i=1}^n \mbox{mult}_{T_i|_{\cls_i}}(\cls_i).
\]
Fix $i \in \{1, \ldots, n\}$ and define $(\lambda_1, \ldots, \lambda_n) \in \mathbb{C}^n$ by $\lambda_j = 0$ if $j = i$ and $\lambda_j = \alpha_j$ if $j \neq i$. From Lemma \ref{lem-spec mult}, it follows that
\[
[P_{\clm_i} G]_{P_{\clm_i} \tilde{T}|_{\clm_i}} = [P_{\clm_i} G]_{(P_{\clm_i} {\tilde{T}_1}|_{\clm_i} - \lambda_1 I_{\clm_i}, \ldots, P_{\clm_i} {\tilde{T}_n}|_{\clm_i} - \lambda_n I_{\clm_i})}.
\]
For simplicity, we denote
\[
\clg_i = [P_{\clm_i} G]_{(P_{\clm_i} {\tilde{T}_1}|_{\clm_i} - \lambda_1 I_{\clm_i}, \ldots, P_{\clm_i} {\tilde{T}_n}|_{\clm_i} - \lambda_n I_{\clm_i})},
\]
in the rest of this section. Also, notice that $\mathbb{C} v_j \perp \mbox{ran} (P_{\clq_j} T_j|_{\clq_j} - \alpha_j I_{\clq_j})$ for all $j = 1, \ldots, n$, and $\mbox{ran} T_i|_{\cls_i} \perp \cls_i \ominus T_i \cls_i$, so that
\[
P_{\cle_i} (P_{\clm_i} {\tilde{T}_j}|_{\clm_i} - \lambda_j I_{\clm_i}) = 0,
\]
for all $j = 1, \ldots, n$, and hence
\[
P_{\cle_i} \clg_i = P_{\cle_i} (\overline{\mbox{span}}\{G\}).
\]
On the other hand, since
\[
P_{\cle} = \bigoplus_{j=1}^n P_{\cle_j},
\]
and $\cle_j \subseteq \clm_j$ for all $j = 1, \ldots, n$, it follows that
\[
P_{\cle} \clg_i= P_{\cle_i} \clg_i.
\]
Hence
\[
\cle  = P_{\cle} \clf = P_{\cle} \Big(\bigoplus_{i=1}^n [P_{\clm_i} G]_{P_{\clm_i} \tilde{T}|_{\clm_i}}\Big) = \bigoplus_{i=1}^n P_{\cle_i} [P_{\clm_i} G]_{P_{\clm_i} \tilde{T}|_{\clm_i}},
\]
that is
\[
\cle = \bigoplus_{i=1}^n P_{\cle_i} \clg_i = \bigoplus_{i=1}^n P_{\cle_i} (\overline{\mbox{span}}\{G\}),
\]
and so
\[
\cle = P_{\cle} (\overline{\mbox{span}}\{G\}).
\]
From this it follows easily that
\[
\begin{split}
\sum_{i=1}^n \mbox{dim} (\cls_i \ominus T_i \cls_i) & = \sum_{i=1}^n \mbox{mult}_{T_i|_{\cls_i}}(\cls_i)
\\
& = \mbox{dim} \cle 
\\
& \leq \mbox{dim} (\overline{\mbox{span}}\{G\})
\\
& = \mbox{dim} (\mbox{span}\{G\}) 
\\
& = \mbox{mult}_{P_{\clf} \tilde{T}|_{\clf}}(\clf),
\end{split}
\]
where the last equality follows from the minimality assumption on $G$.
Therefore, Theorem \ref{thm-rank inequality} implies the following:

\begin{thm}\label{thm-rank inequality2}
Assume the setting of Theorem \ref{thm-rank inequality}. If $\cls_i$ satisfies the generating wandering subspace property with respect to
$T_i|_{\cls_i}$ and $T_i^*|_{\clq_i}$ has non-empty point spectrum for all $i = 1, \ldots, n$, then
\[
\mbox{mult}_{\tilde{T}|_{\cls}}(\cls) \geq \sum_{i=1}^n
\mbox{mult}_{T_i|_{\cls_i}}(\cls_i) = \sum_{i=1}^n \mbox{dim} (\cls_i \ominus T_i \cls_i).
\]
\end{thm}

To proceed further, we note, by Lemma \ref{lemma-old} (or, more specifically \eqref{eq-S Pi Xi}), that
\[
\cls = \sum_{i=1}^n \mbox{ran} \tilde{P}_i.
\]
In addition, let us assume that $\mbox{mult}_{T_i}(\clh_i) = 1$, $i = 1, \ldots, n$. Then
\[
\mbox{mult}_{\tilde{T}|_{\cls}}(\cls) \leq \sum_{i=1}^n \mbox{mult}_{T_i|_{\cls_i}}(\cls_i).
\]
Therefore, by Theorem \ref{thm-rank inequality2}, we have the
main theorem of this paper as:

\begin{thm}\label{thm-main1}
Let $\clh_1, \ldots, \clh_n$ be Hilbert spaces, let $T_i \in \clb(\clh_i)$, and let $\clq_i$ be a $T_i^*$-invariant closed subspace of $\clh_i$,
$i = 1, \ldots, n$. Assume that $T_i|_{\clq_i^\perp} \in \clb(\clq_i^{\perp})$ satisfies the generating wandering subspace property, $T_i^*|_{\clq_i}$ has non-empty point spectrum and that $\mbox{mult}_{T_i}(\clh_i) = 1$ for all $i = 1, \ldots, n$. Then
\[
\mbox{mult}_{\tilde{T}|_{ (\clq_1 \otimes \cdots \otimes
\clq_n)^\perp}} (\clq_1 \otimes \cdots \otimes \clq_n)^\perp
= \sum_{i=1}^n \mbox{mult}_{T_i|_{\clq_i^\perp}}(\clq_i^\perp).
\]
\end{thm}

\section{Applications and Concluding Remarks}

In this section, we complement the main theorem, Theorem \ref{thm-main1}, by some concrete examples and final remarks.

We first explain the notion of zero-based invariant subspaces of reproducing kernel Hilbert spaces. Let $k : \mathbb{D} \times \D \raro \mathbb{C}$ be a positive definite kernel. For each fixed $w \in \D$, let $z \mapsto k(z, w)$ is analytic on $\D$. Suppose $\clh_k \subseteq \clo(\D)$ is the reproducing kernel Hilbert space corresponding to the kernel $k$ and $M_z$, the multiplication operator by the coordinate function $z$, on $\clh_k$ is bounded. Let us further assume that
\[
\ker (M_z^* - \lambda I_{\clh_k}) = \mathbb{C} k(\cdot, \lambda) \quad \quad \quad (\lambda \in \D).
\]
Here $k(\cdot, \lambda)$, for $\lambda \in \D$, denotes the kernel function $z \mapsto k(z, \lambda)$ on $\D$.

A reproducing kernel Hilbert space that satisfies all the properties listed above is called a \textit{regular reproducing kernel Hilbert space}.

It is easy to see that the Dirichlet space, the Hardy, the unweighted Bergman space and the weighted Bergman spaces over $\mathbb{D}$ are regular reproducing kernel Hilbert spaces.

Suppose $\clh_k$ is a regular reproducing kernel Hilbert space. A closed subspace $\cls \subseteq \clh_k$ is called \textit{zero-based invariant subspace} if there exists $\lambda \in \D$ such that $f(\lambda) = 0$ for all $f \in \cls$ and $z \cls \subseteq \cls$.

Now let $\clh_k$ be a regular reproducing kernel Hilbert space, and let $\clq$ be an $M_z^*$-invariant closed subspace of $\clh_k$. Suppose $\lambda \in \D$. Then $M_z^*f = \bar{\lambda} f$ for some non-zero $f \in \clq$ if and only if $f = c k(\cdot, \lambda)$ for some non-zero scalar $c \in \mathbb{C}$. On the other hand, since
\[
\langle g, k(\cdot, \lambda) \rangle = g(\lambda) \quad \quad (g \in \clh_k),
\]
it follows that $k(\cdot, \lambda) \in \clq$ if and only if $g(\lambda) = 0$ for all $g \in \clq^\perp$. We have therefore proved the following:

\begin{propn}
Let $\clh_k$ be a regular reproducing kernel Hilbert space, and let $\clq$ be a closed $M_z^*$-invariant subspace of $\clh_k$. Then $M_z^*|_{\clq}$ has non-empty point spectrum if and only if $\clq^\perp$ is a zero-based invariant subspace of $\clh_k$.
\end{propn}

As an immediate corollary of Theorem \ref{thm-main1} we have now:

\begin{cor}\label{cor-3}
Let $\clh_{k_i}$ be a regular reproducing kernel Hilbert space, $\mbox{mult}_{M_{z}}(\clh_{k_i})=1$, and let $\clq_i$ be a proper closed $M_z^*$-invariant subspace of $\clh_{k_i}$, $i = 1, \ldots, n$. If $\clq_i^{\bot}$ is a zero-based invariant subspace of $\clh_{k_i}$ such that
\[
\mbox{dim} (\clq_i^\perp \ominus z \clq_i^\perp) < \infty,
\]
for all $i = 1, \ldots, n$, then
\[
\mbox{mult}_{M_{\bm z}|_{ (\clq_1 \otimes \cdots \otimes \clq_n)^\perp}}
(\clq_1 \otimes \cdots \otimes \clq_n)^\perp = \sum_{i=1}^n
(\mbox{mult}_{M_z|_{\mathcal{Q}_i^\perp}} (\mathcal{Q}_i^{\bot})) = \sum_{i=1}^n
\mbox{dim} (\clq_i^\perp \ominus z \clq_i^\perp).
\]
\end{cor}

Now let $\clh_{k_i}$ be the Hardy space or the Dirichlet space over
$\mathbb{D}$, and let $\clq_i$ be a non-zero shift co-invariant
(that is, $M_z^*$-invariant) subspace of $\clh_{k_i}$. By \cite{AB} and
\cite{SR}, $M_z|_{\clq_i^\perp}$ satisfies the generating wandering subspace
property and the dimension of the generating wandering subspace is one, that is
\[
\mbox{dim} (\clq_i^\perp \ominus z \clq_i^\perp) = 1,
\]
for all $i = 1, \ldots, n$. Then, in view of Theorem \ref{thm-main1} (and \cite{T}) we have the following:

\begin{cor}\label{cor-1}
Let $\clh_{k_i}$, $i = 1, \ldots, n$, denote either the Hardy space or the Dirichlet space over $\mathbb{D}$. Suppose $\clq_i$ is a proper closed $M_z^*$-invariant subspaces of $\clh_{k_i}$, $i = 1, \ldots, n$. If $\clq_i^{\bot}$ is a zero-based $M_{z}$-invariant subspace of $\clh_{k_i}$, $i=1,\ldots,n$, then,
\[
\mbox{mult}_{M_{\bm z}|_{ (\clq_1 \otimes \cdots \otimes \clq_n)^\perp}}
(\clq_1 \otimes \cdots \otimes \clq_n)^\perp = n.
\]
\end{cor}

A similar argument and the generating wandering subspace property of
shift invariant subspaces of the Bergman space \cite{ARS} yields the
following:

\begin{cor}\label{cor-2}
Let $\clh_{k_i}$, $i = 1, \ldots, n$, be the Dirichlet space, the
Bergman space or the Hardy space over $\mathbb{D}$. Let $\clq_i$, $i
= 1, \ldots, n$, be proper closed shift co-invariant subspaces of
$\clh_{k_i}$. If $\clq_i^{\bot}$ is a zero based $M_{z}$-invariant subspace of $\clh_{k_i}$ and
\[
\mbox{dim} (\clq_i^\perp \ominus z \clq_i^\perp) < \infty,
\]
for all $i = 1, \ldots, n$, then
\[
\mbox{mult}_{M_{\bm z}|_{ (\clq_1 \otimes \cdots \otimes \clq_n)^\perp}}
(\clq_1 \otimes \cdots \otimes \clq_n)^\perp = \sum_{i=1}^n
(\mbox{mult}_{M_z|_{\mathcal{Q}_i^\perp}} (\mathcal{Q}_i^{\bot})) = \sum_{i=1}^n
\mbox{dim} (\clq_i^\perp \ominus z \clq_i^\perp).
\]
\end{cor}

Note that the generating wandering subspace assumption in Corollary
\ref{cor-2} ensures that (see Proposition \ref{prop-wwrank})
\[
\mbox{mult}_{M_z|_{\clq_i^\perp}}(\clq_i^\perp) < \infty,
\]
for all $i = 1, \ldots, n$. At present it is not very clear whether
the generating wandering subspace assumption can be replaced by
finite multiplicity property. Our methods rely heavily on the
assumption that the invariant subspaces are zero-based and satisfies the generating wandering subspace property.

\vspace{4mm}

\noindent\textbf{Acknowledgement:} The second author is supported in
part by NBHM (National Board of Higher Mathematics, India) grant
NBHM/R.P.64/2014, and the Mathematical Research Impact Centric
Support (MATRICS) grant, File No : MTR/2017/000522, by the Science
and Engineering Research Board (SERB), Department of Science \&
Technology (DST), Government of India. The third author is supported by the
Department of Atomic Energy (DAE) through the NBHM Postdoctoral fellowship and
acknowledges Indian Statistical Institute, Bangalore, for warm hospitality.

\end{document}